\documentclass[leqno]{amsart}

\usepackage[english]{babel} 									
\usepackage[T1]{fontenc}							
\usepackage[utf8]{inputenx}
				
\usepackage{mathtools}								
\usepackage{empheq}									
\usepackage{enumitem}								

\usepackage{amssymb}								
\usepackage{mathrsfs}								
\usepackage{bm}                                     
				
\usepackage{verbatim}                               
\usepackage{iflang}									
\usepackage{xifthen}								

\usepackage[final]{pdfpages}						
\usepackage{todonotes}								
\usepackage{aliascnt}								
\usepackage{needspace}								
\usepackage{refcount}

\usepackage{subcaption}			
\usepackage{placeins} 			
\usepackage{grffile}			
\usepackage[all]{xy}			
\usepackage{stmaryrd}
\usepackage{comment}

\usepackage{transparent}
\usepackage{color}				

\usepackage{tikz}               
\usetikzlibrary{math,arrows.meta, cd}
\DeclareGraphicsExtensions{.pdf,.png,.jpg}

\usepackage{graphicx}			
 
\usepackage{diagbox} 

\usepackage[babel]{csquotes}							

\usepackage{nameref}									
\usepackage[colorlinks	=	true,						
          	 	linkcolor	=	red,
            	urlcolor	=	red,
            	citecolor	=	red]%
            	{hyperref}
\usepackage{bookmark}									
\usepackage{cleveref}									
\usepackage{microtype}

\numberwithin{equation}{section}

\newtheorem{theorem}{Theorem}[section]

\newtheorem{lemma}[theorem]{Lemma}
\newtheorem{proposition}[theorem]{Proposition}

\newtheorem*{theorem*}{Theorem}

{\bf}{\it}

\theoremstyle{remark}

\newtheorem*{remark*}{Remark}

\theoremstyle{remark}

\newcommand{\N}{\mathbb{N}}
\renewcommand{\S}{\mathbb{S}}

\newcommand{\R}{\mathbb{R}}
\newcommand{\Q}{\mathbb{Q}}
\newcommand{\Z}{\mathbb{Z}}

\newcommand{\cH}{\mathcal{H}}
\newcommand{\cL}{\mathcal{L}}
\newcommand{\st}{\textup{st}}

\DeclareMathOperator{\sys}{sys}

\DeclareMathOperator{\diam}{diam}

\DeclareMathOperator{\im}{Im}
\DeclareMathOperator{\interior}{int}

\allowdisplaybreaks

\newcommand{\norm}[1]{ \left\Vert #1 \right\Vert }	
\newcommand{\apmd}[2][]{							
	\ifthenelse{\equal{#1}{}}%
					{ \operatorname{N}_{#2}	}%
					{ \operatorname{N}_{#1,#2} 	}}

\newcommand{\aint}[2][]{
	\ifthenelse{\equal{#1}{}}%
					{%
\mathchoice%
      {\mathop{\kern 0.2em\vrule width 0.6em height 0.69678ex depth -0.58065ex
              \kern -0.8em \intop}\nolimits_{\kern -0.45em#2}^{#1}}%
      {\mathop{\kern 0.1em\vrule width 0.5em height 0.69678ex depth -0.60387ex
              \kern -0.6em \intop}\nolimits_{#2}^{#1}}%
      {\mathop{\kern 0.1em\vrule width 0.5em height 0.69678ex depth -0.60387ex
              \kern -0.6em \intop}\nolimits_{#2}^{#1}}%
      {\mathop{\kern 0.1em\vrule width 0.5em height 0.69678ex depth -0.60387ex
              \kern -0.6em \intop}\nolimits_{#2}^{#1}}}%
					{%
\mathchoice%
      {\mathop{\kern 0.2em\vrule width 0.6em height 0.69678ex depth -0.58065ex                                              
              \kern -0.8em \intop}\nolimits_{\kern -0.45em#1}^{#2}}%
      {\mathop{\kern 0.1em\vrule width 0.5em height 0.69678ex depth -0.60387ex
              \kern -0.6em \intop}\nolimits_{#1}^{#2}}%
      {\mathop{\kern 0.1em\vrule width 0.5em height 0.69678ex depth -0.60387ex
              \kern -0.6em \intop}\nolimits_{#1}^{#2}}%
      {\mathop{\kern 0.1em\vrule width 0.5em height 0.69678ex depth -0.60387ex
              \kern -0.6em \intop}\nolimits_{#1}^{#2}}}}

\title[Systolic inequalities for metric surfaces]{Systolic inequalities for metric surfaces via filling minimality}

\thanks{The authors were supported by the Swiss National Science Foundation grant 212867.}

\author{Toni Ikonen}
\author{Denis Marti}
\author{Noa Vikman}

\address{Department of Mathematics\\ University of Fribourg\\  Chemin du Mus\'ee 23\\  1700 Fribourg, Switzerland}
\email{toni.ikonen@unifr.ch}
\address{Department of Mathematics\\ University of Fribourg\\  Chemin du Mus\'ee 23\\  1700 Fribourg, Switzerland}
\email{denis.marti@unifr.ch}
\address{Department of Mathematics\\ University of Fribourg\\  Chemin du Mus\'ee 23\\  1700 Fribourg, Switzerland}
\email{noa.vikman@unifr.ch}

\begin{document}

\begin{abstract}
    We prove optimal systolic inequalities for length spaces homeomorphic to a torus of genus one or a real projective plane. In both cases, the optimal constant coincides with the constant from the (reversible) Finsler setting. This generalizes the classical results for Riemannian and Finsler surfaces.

    The proof of the inequality for the torus relies on an analysis of the asymptotic volume growth of the universal cover, together with a strengthening of the previously known area minimality of two-dimensional normed planes. For the inequality of the real projective plane, we similarly extend a minimality result of hemispheres. These results build upon works of Burago--Ivanov and Ivanov, respectively. In their proofs we apply recent uniformization theorems for metric disks and the theory of area-minimizing disks in metric spaces due to Lytchak--Wenger.

\end{abstract}

\maketitle

\section{Introduction}

\subsection{Background}
    The systole of a Riemannian manifold $(M,g)$ is defined as the smallest length of a non-contractible closed curve in $M$. That is,
    $$\sys(M,g) = \inf\{ \ell(\gamma) \colon \gamma \textup{ is a non-contractible closed curve in }M\}.$$
    
    Clearly, this definition directly extends to Finsler manifolds $(M,F)$. The central theme of systolic geometry is the relationship between the systole of a space and its volume in terms of so-called systolic inequalities. The first optimal systolic inequality was proven for the $2$-dimensional (genus one) torus in unpublished work by Loewner.

    \medskip

    \textbf{Loewner's systolic inequality.} Let $(\mathbb{T}, g)$ be a $2$-dimensional torus with a Riemannian metric. Then,
    $$\textup{Area}(\mathbb{T}, g) \geq  \frac{\sqrt{3}}{2} \sys(\mathbb{T}, g)^2.$$
    Equality holds if and only if $(\mathbb{T}, g)$ is isometric to the quotient of the Euclidean plane by some hexagonal lattice.      

    \medskip
    
    The proof is a classical application of the uniformization theorem, which reduces the situation to the flat case, where the inequality follows from elementary methods; see e.g. \cite[Theorem 4]{Balacheff-Gil-Moreno-de-Mora-finsler-tori}. Loewner's systolic inequality first appeared in the work of his student Pu \cite{Pu}, where the optimal systolic inequality for the projective plane was proven. 

    \medskip
    
    \textbf{Pu's systolic inequality.} Let $(\mathbb{RP}^2, g)$ be a $2$-dimensional real projective plane with a Riemannian metric. Then,
    $$\textup{Area}(\mathbb{RP}^2, g) \geq  \frac{2}{\pi} \sys(\mathbb{RP}^2, g)^2.$$
    Equality holds if and only if the metric $g$ has constant curvature.  

    \medskip

    Together with the systolic inequality for the Klein bottle \cite{Bavard}, these are the only known optimal systolic inequalities for general closed Riemannian manifolds; see also \cite{Jabbour-Sabourau-sphere}. If additional assumptions are imposed on the space, for example, on the curvature, then more can be said; see e.g. \cite{Ivanov-katz-loewner,Jabbour-Sabourau-negatively,felix-1984-on-the-first-eigenvalue-of-the-Laplace-operator-on-selected-examples-of-compact-Riemann-surfaces,parlier-2008-fixed-point-free-involutions-on-riemann-surfaces,Katz-book, katz-sabourau-cat0,Katz-Sabourau-extremal,katz-Sabourau-Nonpositively}.
    
    \medskip

    We now discuss known results in the (reversible) Finsler setting. 
    In this case, the optimal systolic inequalities depend on the choice of Finsler area (volume). 
    The most prominent ones are the Busemann--Hausdorff area and the Holmes--Thompson area; for these, the optimal systolic inequalities are only known for the projective plane \cite{ivanov-systolic-ineq-proj} and the torus \cite{Balacheff-Gil-Moreno-de-Mora-finsler-tori, Sabourau-finsler-tori}. Additionally, there are partial results for the Klein bottle \cite{Sabourau-finsler-klein}. The optimal systolic inequality with respect to the Busemann--Hausdorff area $\textup{Area}_\mathrm{BH}(\mathbb{T},F)$ was recently proved for Finslerian tori; cf. \cite[Theorem 1]{Balacheff-Gil-Moreno-de-Mora-finsler-tori}.

    \medskip
    
    \textbf{Systolic inequality for Finsler tori.} Let $(\mathbb{T}, F)$ be a $2$-dimensional torus with a Finsler metric. Then,
    $$\textup{Area}_\mathrm{BH}(\mathbb{T}, F) \geq  \frac{\pi}{4} \sys(\mathbb{T}, F)^2.$$
    Equality holds for the flat metric corresponding to the supremum norm.    
    
    \medskip
    
    We also recall the following Finsler version of Pu's systolic inequality due to Ivanov \cite[Theorem 3]{ivanov-systolic-ineq-proj}.
    
    \medskip 
    
    \textbf{Systolic inequality for Finsler real projective planes.} Let $(\mathbb{RP}^2, F)$ be a real projective plane $\mathbb{RP}^2$ equipped with a Finsler metric. Then,
        $$\textup{Area}_\mathrm{BH}(\mathbb{RP}^2, F) \geq  \frac{2}{\pi} \sys(\mathbb{RP}^2, F)^2.$$
        Equality holds if and only if $F$ is a Riemannian metric of constant curvature.

    \medskip

    Another fundamental systolic inequality that applies to manifolds in every dimension was proven by Gromov \cite{gromov-filling}. It states that any essential, closed, connected $n$-dimensional Riemannian manifold admits a systolic inequality with a constant depending only on $n$. Recently, Gromov's systolic inequality has been generalized to metric spaces \cite{Filling-metric-spaces} using an approach different from Gromov's original proof, based on the relationship between the Uryson width of a space and its Hausdorff content; see also \cite{Papasoglu-Uryson}. Interestingly, this approach also led to an improvement of the constant in Gromov's systolic inequality for Riemannian manifolds \cite{Nabutosvsky-linear}. Continuing this line of research, we prove metric versions of the optimal systolic inequalities for tori and real projective planes.

\subsection{Main results}
    The systole can be defined in metric spaces as well. Indeed, if $(X,d)$ is a metric space, we define the systole by
    \begin{equation}
        \sys(X,d) = \inf\big\{ \ell(\gamma) \colon \gamma \textup{ is a non-contractible closed curve in }X\big\}.
    \end{equation}
    Notice that every Riemannian or Finsler manifold can be understood as a metric space by passing to the intrinsic length distance $d_{\mathrm{int}}$, and the various definitions of systole coincide in this setting. Our aim is to extend systolic inequalities to length spaces homeomorphic to a $2$-dimensional torus or the real projective plane, and thus, we need a notion of area. To achieve this, we equip such a space with the two-dimensional Hausdorff measure $\mathcal{H}_d^2$, and write $\mathrm{Area}(X,d) = \mathcal{H}^{2}_{d}( X )$. We recall that the two-dimensional Hausdorff measure coincides with the Riemannian volume on a Riemannian manifold and the Busemann--Hausdorff area of a Finsler manifold. In particular, we can reformulate the systolic inequalities for Riemannian or Finsler surfaces by interpreting them as length spaces, as described above.
    
    We now formulate our first main result.

    \begin{theorem}\label{theorem-main-result}
        Let $(\mathbb T,d)$ be a length space homeomorphic to a $2$-dimensional torus. Then
        \begin{align*}
           \textup{Area}(\mathbb{T},d) \geq  \frac{\pi}{4} \sys( \mathbb{T},d )^2.
        \end{align*}
        Equality holds if $(\mathbb T,d)$ is isometric to $\R^2/\mathbb{Z}^2$, equipped with the supremum norm.
    \end{theorem}
    We emphasize that this result directly generalizes the systolic inequality for Finsler tori.
    
    Our second main theorem extends Pu's systolic inequality to the metric setting.
    \begin{theorem}\label{theorem-main-result-real-projective-plane}
        Let $(\mathbb{RP}^2,d)$ be a length space homeomorphic to the real projective plane. Then
        \begin{align*}
           \textup{Area}(\mathbb{RP}^2,d)\geq  \frac{2}{\pi} \sys( \mathbb{RP}^2,d )^2.
        \end{align*}
        Equality holds if $(\mathbb{RP}^2,d)$ is isometric to the real projective plane equipped with a Riemannian metric of constant curvature.
    \end{theorem}

    A key step in the proofs of \Cref{theorem-main-result,theorem-main-result-real-projective-plane} is characterizing suitable minimal metrics on the universal cover of the torus and the real projective plane, respectively. Such a classification is well understood in the Finsler setting \cite{burago-ivanov-minimality-planes,Ivanov-lipschitz-metrics}. We extend the characterization of minimal metrics by applying recent advances in the uniformization of metric surfaces \cite{ntalampekos-romney-2026-polyhedral-approximation-for-non-length,ntalampekos-romney-length,Meier-Wenger}, in combination with the theory of area-minimizing disks for metric targets due to Lytchak--Wenger \cite{lychak-wenger-20217-area-minimizing-discs-in-metric-spaces,lytchak-wenger-2018-intrinsic-structure-of-minimal-discs-in-metric-spaces}.
    
    We first formulate an extension of the minimality of two-dimensional planes by Burago--Ivanov \cite{burago-ivanov-minimality-planes,Burago-Ivanov-asymptotic-volume-finsler-tori}, which is an important ingredient in the proof of the systolic inequality for metric tori.
    
    \begin{proposition}\label{proposition-minimality-of-two-dimensional-planes}
        Let $\mathbb{W}$ be a Banach space and $\mathbb{V} \subset \mathbb{W}$ a two-dimensional subspace. Then any Jordan domain $\Omega \subset \mathbb{V}$ with rectifiable boundary is minimal in the following sense: For any two-dimensional metric disk $X$ and any locally $1$-Lipschitz map $v \colon X \to \mathbb{W}$, where $v|_{ \partial X } \colon \partial X \to \partial \Omega$ is a uniform limit of homeomorphisms, it holds that $\mathcal{H}^{2}_{X}( X ) \geq \mathcal{H}^{2}_{ \mathbb{W} }( \Omega )$.
    \end{proposition}
    In this manuscript, a metric space $X$ is called a \emph{metric disk} if $X$ is homeomorphic to a closed disk in the plane. 

    Our second minimality result concerns the minimality of Finsler hemispheres due to Ivanov \cite{ivanov-systolic-ineq-proj,Ivanov-lipschitz-metrics}. We prove an extension of his results for metric disks and use it to prove the systolic inequality for metric projective planes.
    
    \begin{proposition}\label{proposition-filling-minimality-of-circles}
        Let $X$ be a $2$-dimensional metric disk for which there exists a 1-Lipschitz homeomorphism $\partial X \to ( \mathbb{S}^1, d_{\mathrm{int}} )$. Then
        \begin{align*}
            \mathcal{H}^{2}(X) \geq 2\pi.
        \end{align*}
        Equality holds if $X$ is isometric to a hemisphere of unit radius.
    \end{proposition}
    Here, $d_{\mathrm{int}}$ refers to the intrinsic length distance on $\S^1$.

\subsection{Structure of the paper}

    The article is structured as follows. In Section \ref{sec: pre}, we fix our notation and present the basic definitions. In Section \ref{section-minimality-results}, we prove the minimality of two-dimensional planes, Proposition \ref{proposition-minimality-of-two-dimensional-planes}, and hemispheres, Proposition \ref{proposition-filling-minimality-of-circles}. In Section \ref{section-systolic-inequality-torus}, we first recall the necessary prerequisites on asymptotic geometry. Then, following the broad strategy of Burago--Ivanov \cite{Burago-Ivanov-asymptotic-volume-finsler-tori}, we prove the systolic inequality for the torus while addressing the challenges of the metric setting. In Section \ref{section-systolic-inequality-real-projective-plane}, we establish Pu's inequality for metric real projective planes.

\subsection{Acknowledgements}
    The authors thank Alexey Balitskiy, Alexander Lytchak, and Hugo Parlier for discussions regarding systolic inequalities, and Stefan Wenger for feedback on an earlier version of the manuscript.

\section{Preliminaries}\label{sec: pre}
    Let $(X,d)$ be a metric space. Given $x \in X$, we write 
    $$B(x,r)=\{y \in X \colon d(x,y)< r\}$$ for the open ball with center $x$ and radius $r>0$. We call a map $f\colon X \to Y$ between metric spaces \textit{$L$-Lipschitz} if 
    $$d(f(x),f(y)) \leq L d(x,y)$$
    for all $x,y \in X$. In case that $f$ is injective and its inverse also is $L$-Lipschitz, we say that $f$ is \textit{$L$-bi-Lipschitz}. For two maps $f, g\colon X \to Y$, we define the \textit{uniform distance} between \(f\) and \(g\) by
    $$ d_\infty(f, g)=\sup_{x\in X} d(f(x), g(x)).$$
    We say that $f_n \rightarrow f$ \emph{uniformly} if $d_{\infty}(f_n,f) \rightarrow 0$.
    
    A \emph{curve} is a continuous map $\gamma \colon I \rightarrow X$ where $I$ is an interval or the unit circle $\mathbb{S}^1$. In the case of an interval, the length of $\gamma$ is defined by
    \begin{align*}
        \ell( \gamma )
        =
        \sup\left\{
        \sum_{i=0}^{k-1} d( \gamma(t_i), \gamma(t_{i+1}) )
        \colon t_i \in I
        \,\text{and}\,
        t_0 < t_1 < \dots < t_k
        \right\}
    \end{align*}
    over finite partitions. We say $\gamma$ is \emph{rectifiable} if $\ell( \gamma ) < \infty$. A curve $\gamma \colon I \to X$ has constant speed if there exists a constant $s > 0$ such that, for every interval $[a,b] \subset I$, $\ell( \gamma|_{[a,b]} ) = s |b-a|$. Analogous definitions apply in the case $I = \mathbb{S}^1$.

    We call a metric space $X$ a \textit{length space} if, for all $x,y\in X$, the distance $d(x,y)$ is equal to the infimal length of curves in $X$ from $x$ to $y$, and we call $X$ \textit{geodesic} if this infimum is attained by such a curve. If $X$ is a compact or proper length space, it is also geodesic by the metric Hopf--Rinow theorem.
    
    The $n$-dimensional Hausdorff measure on $X$ is defined as follows: For a set $A \subset X$, we have
    \begin{align*}
        \mathcal{H}^{n}_X( A ) = \sup_{ \delta > 0 } \mathcal{H}^{n}_{X,\delta}(A),
    \end{align*}
    where
    \begin{align*}
        \mathcal{H}^{n}_{X,\delta}(A)
        =
        \frac{ \omega_n }{ 2^n }
        \inf\left\{
            \sum_{ i } \diam(E_i)^{n}
            \colon
            \text{ $A \subset \bigcup E_i$ and $\diam E_i < \delta$ for $i \in \N$ }
        \right\}.
    \end{align*}
    We often equip the same space with several metrics, in which case we replace $X$ in the subscript with the specific metric we use. When the choice of ambient space is understood, we omit the subscript $X$ from the notation. Here, $\omega_n$ is the Lebesgue measure of the $n$-dimensional Euclidean unit ball. The normalization guarantees that $\mathcal{H}^n$ on $\R^n$ coincides with the $n$-dimensional Lebesgue measure $\cL^n$.
    
    Let $\norm{\cdot}$ be a norm on $\R^2$. Since $\Z^2$ acts as isometries on $( \R^2, \norm{\cdot} )$, we obtain a \emph{flat Finsler torus} $(\R^2/\Z^2,\norm{\cdot})$ associated with $\norm{\cdot}$. The induced length distance between two orbits $\Z^2 \cdot x,\Z^2 \cdot y\in (\R^2/\Z^2,\norm{\cdot})$ is equal to 
    $$\min_{z\in \Z^2}\norm{x-y-z}.$$
    The Busemann--Hausdorff area of $(\R^2/\Z^2,\norm{\cdot})$ is defined by
    $$\textup{Area}_\mathrm{BH}(\R^2/\Z^2,\norm{\cdot}) = \int_{[0,1]^2} \frac{\pi}{\cL^2(K)} \,d\cL^2$$
    where $K \subset \R^2$ denotes the unit disk of $\norm{\cdot}$. It is an immediate consequence of the area formula that $\textup{Area}_\mathrm{BH}(\R^2/\Z^2,\norm{\cdot})$ is equal to the Hausdorff $2$-measure of $(\R^2/\Z^2,\norm{\cdot})$ equipped with the induced length distance. We refer to \cite{Balacheff-Gil-Moreno-de-Mora-finsler-tori} for more information on Finsler tori.

\section{Minimality of disks and hemispheres}\label{section-minimality-results}
    In this section, we prove Propositions \ref{proposition-minimality-of-two-dimensional-planes} and \ref{proposition-filling-minimality-of-circles}. The proofs apply metric-valued Sobolev theory. We refer to \cite{lychak-wenger-20217-area-minimizing-discs-in-metric-spaces} for the definitions and only recall the necessary facts. For each Sobolev map $u \in W^{1,2}( U, X )$, defined on a Lipschitz domain $U \subset \R^2$, there is a natural notion of parametrized (Busemann--Hausdorff) area $\mathrm{Area}(u)$, Reshetnyak energy $E(u)$, and boundary trace $\mathrm{trace}( u ) \colon \partial U \to X$.

    Below we use the fact that if $v \colon X \to Y$ is a locally $1$-Lipschitz map, then $v \circ u \in W^{1,2}( U, Y )$ and $\mathrm{Area}( v \circ u ) \leq \mathrm{Area}( u )$. We also note that if $\phi \colon V \to U$ is a conformal homeomorphism between domains, then $u \circ \phi \in W^{1,2}( V, X )$ with area and energy equal to that of $u$.  We emphasize that when $X$ is a Riemannian manifold and $u$ is a Lipschitz map, then $\mathrm{Area}(u)$ corresponds to the integral of the usual Jacobian; for Sobolev maps and more general targets, the parametrized area is also defined by integration of a generalized Jacobian.

    A Jordan curve $\Gamma \subset X$ always refers to a curve homeomorphic to $\mathbb{S}^1$. The class $\Lambda( \Gamma, X )$ consists of all $u \in W^{1,2}( \mathbb{D}, X )$ for which $\mathrm{trace}(u)$ has a continuous representative that is a uniform limit of homeomorphisms $\partial \mathbb{D} \to \Gamma$; we recall that $\mathbb{D} \subset \R^2$ is the Euclidean unit disk. The \emph{filling area} of $\Gamma$ is defined as
    \begin{align*}
        \mathrm{FillArea}_{X}( \Gamma )
        =
        \inf_{ u \in \Lambda( \Gamma, X ) } \mathrm{Area}(u).
    \end{align*}
    This is closely related to the \emph{Lipschitz filling area}:
    \begin{align*}
        \mathrm{FillArea}^{\mathrm{LIP}}_{X}( \Gamma )
        =
        \inf_{ \substack{ u \in \Lambda( \Gamma, X ) \\ \text{$u$ is Lipschitz} } } \mathrm{Area}(u).
    \end{align*}

    Obviously, \emph{a priori}, the Lipschitz filling area is at least the Sobolev one. However, when the target is a Banach space, or, more generally, Lipschitz $1$-connected up to some scale (see \cite[page 81, third paragraph]{lytchak-wenger-young-2020-dehn-functions-and-holder-extensions-in-asymptotic-cones} for the definition), these quantities coincide; see also \cite[Section 10]{lytchak-wenger-2018-intrinsic-structure-of-minimal-discs-in-metric-spaces}. We record this statement as follows.
    
    \begin{lemma}\label{lemma-sobolev-to-lipschitz-filling-area-lipschitz-connectivity}
        Let $X$ be a complete metric space that is Lipschitz $1$-connected up to some scale and let $\Gamma \subset X$ be a rectifiable Jordan curve. Then
        \begin{align*}
            \mathrm{FillArea}_{X}( \Gamma ) = \mathrm{FillArea}^{\mathrm{LIP}}_{X}( \Gamma ).
        \end{align*}
    \end{lemma}
    
    \begin{proof}
    Consider $u \in \Lambda( \Gamma, X )$. By attaching a Sobolev homotopy of arbitrarily small area to $u$, we obtain $v \in \Lambda( \Gamma, X )$ such that the boundary trace of $v$ is a constant-speed Lipschitz parametrization of $\Gamma$; cf. \cite[Lemma 4.8]{lytchak-wenger-2018-intrinsic-structure-of-minimal-discs-in-metric-spaces}. Now, in the proof of \cite[Proposition 3.1]{lytchak-wenger-young-2020-dehn-functions-and-holder-extensions-in-asymptotic-cones}, it is established that $v$ can be approximated arbitrarily well in area by Lipschitz maps that have the same boundary trace as $v$. Both steps use the Lipschitz $1$-connectivity in a critical way. By combining these facts, the claimed equality follows.       
    \end{proof}
    We also use the following generalization of the classical uniformization theorem due to \cite[Theorem 1.4]{ntalampekos-romney-2026-polyhedral-approximation-for-non-length}. Note that \cite[Theorem 1.4]{ntalampekos-romney-2026-polyhedral-approximation-for-non-length} is a stronger statement but we only recall the properties we need.
    
    \begin{proposition}\label{proposition-uniformization-theorem}
        Let $X$ be a metric space homeomorphic to the closed disk and having finite two-dimensional Hausdorff measure. Then there exists $u \in \Lambda( \partial X, X )$ that is a uniform limit of homeomorphisms $\overline{\mathbb{D}} \to X$.
    \end{proposition}

\subsection{Proof of \Cref{proposition-minimality-of-two-dimensional-planes}}\label{sec: minimality}

    We need the following extension of the results by Burago--Ivanov \cite[Theorem 1]{burago-ivanov-minimality-planes} (see also \cite{Burago-Ivanov-asymptotic-volume-finsler-tori}), proving that the Busemann--Hausdorff area induces a quasi-convex area density. This implies the following result, well known to experts.
    \begin{proposition}\label{proposition-minimality-of-two-dimensional-planes-affine}
    Let $\mathbb{W}$ be a Banach space and $\mathbb{V} \subset \mathbb{W}$ a two-dimensional subspace. Suppose that $\Gamma \subset \mathbb{V}$ is a rectifiable Jordan curve and $\Omega \subset \mathbb{V}$ the Jordan domain bounded by $\Gamma$. Then
    \begin{align*}
        \mathrm{FillArea}_{\mathbb{W}}( \Gamma ) = \mathcal{H}^{2}( \Omega ).
    \end{align*}
    \end{proposition}
    \begin{proof}
       We start with the elementary observation:
        \begin{align*}
            \mathrm{FillArea}_{ \mathbb{W} }( \Gamma ) \leq \mathrm{FillArea}_{ \mathbb{V} }( \Gamma ) \leq \mathcal{H}^{2}( \Omega ).
        \end{align*}
        To see this, consider an invertible linear map $L \colon \R^2 \to \mathbb{V} \subset \mathbb{W}$ and a conformal homeomorphism $\phi \colon \mathbb{D} \to L^{-1}( \Omega )$ obtained from the Riemann mapping theorem. It follows that
        \begin{align*}
            \mathrm{FillArea}_{ \mathbb{W} }( \Gamma ) \leq \mathrm{Area}( L \circ \phi ) = \mathcal{H}^{2}(\Omega).
        \end{align*}
        We prove that the inequality above is an equality by reducing the proof to a finite-dimensional statement as follows.
        
        Consider a separable dense set $\{ v_1,\dots,v_{j},v_{j+1},\dots\}$ in the unit sphere of $\mathbb{V}$ and consider, for each $j \in \N$, $w_j \in \mathbb{W}^{*}$ with unit norm and $w_j(v_j) = 1$. Next, we consider the linear map $\iota \colon \mathbb{W} \to \ell^{\infty}$ defined by $v \mapsto ( w_1(v), \dots, w_{j}(v), w_{j+1}(v), \dots )$, where $\ell^{\infty}$ is the Banach space of bounded sequences equipped with the supremum norm. It is immediate that $\iota$ is $1$-Lipschitz and its restriction to $\mathbb{V}$ is an isometry. In particular,
        \begin{align*}
            \mathrm{FillArea}_{ \ell^{\infty} }( \iota(\Gamma) ) \leq \mathrm{FillArea}_{ \mathbb{W} }( \Gamma ).
        \end{align*}
        It suffices to prove $\mathrm{FillArea}_{ \ell^{\infty} }( \iota(\Gamma) ) \geq \mathcal{H}^{2}(\Omega)$ to finish. To this end, consider the linear projection $\pi_n \colon \ell^{\infty} \to ( \mathbb{R}^n, \| \cdot \|_{\infty} )$, where $\pi_n( x_1, \dots, x_{j}, x_{j+1}, \dots ) = ( x_1, \dots, x_{n-1}, x_{n} )$. That is, $\pi_n$ truncates the sequence at index $n$, and we identify the target with the $n$-dimensional Euclidean space equipped with the supremum norm. Clearly, there exists $n_0 \in \N$ such that for every $n \geq n_0$, the composition $\pi_n \circ \iota|_{ \mathbb{V} }$ has rank two. In particular, $\Gamma_n = \pi_n \circ \iota( \Gamma )$ is a rectifiable Jordan curve in the Jordan domain $\Omega_n = \pi_n \circ \iota(\Omega)$ inside the two-dimensional subspace $\mathbb{V}_n = \pi_n \circ \iota(\mathbb{V})$ of $( \mathbb{R}^n, \| \cdot \|_{\infty} )$. By construction, it holds that
        \begin{align*}
            \mathrm{FillArea}_{ ( \mathbb{R}^n, \| \cdot \|_{\infty} ) }( \Gamma_n ) \leq \mathrm{FillArea}_{ \ell^{\infty} }( \iota(\Gamma) ).
        \end{align*}
        By the minimality of two-dimensional planes in finite dimensions proved in \cite{burago-ivanov-minimality-planes} (see also \cite[Conjecture B]{Burago-Ivanov-asymptotic-volume-finsler-tori}) and by \Cref{lemma-sobolev-to-lipschitz-filling-area-lipschitz-connectivity}, it follows that
        \begin{align*}
            \mathcal{H}^{2}( \Omega_n ) = \mathrm{FillArea}^{\mathrm{LIP}}_{  ( \mathbb{R}^n, \| \cdot \|_{\infty} )  }( \Gamma_n ) = \mathrm{FillArea}_{ ( \mathbb{R}^n, \| \cdot \|_{\infty} ) }( \Gamma_n ).
        \end{align*}
        We observe that, for every $\varepsilon > 0$, there exists $n_\varepsilon \in \N$ so that $\pi_n \circ \iota|_{ \mathbb{V} }$ is $(1+\varepsilon)$-bi-Lipschitz for every $n \geq n_\varepsilon$. This implies that
        \begin{align*}
            \lim_{ n \to \infty } \mathcal{H}^{2}( \Omega_n ) = \mathcal{H}^{2}( \Omega ).
        \end{align*}
        Combining this fact with the chains of inequalities completes the proof.
    \end{proof}
    
    We now proceed with the proof of the main result of this subsection.
    \begin{proof}[Proof of \Cref{proposition-minimality-of-two-dimensional-planes}]
    In case $\mathcal{H}^{2}_X( X ) = \infty$, there is nothing to prove. In the finite case, we apply \Cref{proposition-uniformization-theorem} and consider $u \in \Lambda( \partial X, X )$ that is a uniform limit of homeomorphisms $\overline{\mathbb{D}} \to X$. Then $v \circ u \in \Lambda( \partial \Omega, \mathbb{W} )$ by assumptions on $v$. Thus, \Cref{proposition-minimality-of-two-dimensional-planes-affine} implies
    \begin{align*}
        \mathcal{H}^{2}( \Omega ) \leq \mathrm{Area}( v \circ u ) \leq \mathrm{Area}(u).
    \end{align*}
    Given that $u$ is Sobolev, the area formula for $u$ yields the existence of a negligible set $N \subset \mathbb{D}$ such that
    \begin{align*}
        \infty > \mathrm{Area}(u) = \int_{ X \setminus u(N) } \sharp( u^{-1}(x) ) \,d\mathcal{H}^{2}(x),
    \end{align*}
    where $\sharp A$ is equal to $\infty$ for infinite sets and the cardinality for finite sets. Since $u$ is a uniform limit of homeomorphisms, it follows that either $u^{-1}(x)$ is a point or a compact and connected set with positive diameter. In the latter case, clearly $\sharp( u^{-1}(x) ) = \infty$ so necessarily $\sharp( u^{-1}(x) ) = 1$ for $\mathcal{H}^2$--almost every $x \in X \setminus u(N)$. Hence
    \begin{align*}
         \mathrm{Area}(u)
         =
         \mathcal{H}^{2}( X \setminus u(N) )
         \leq
         \mathcal{H}^{2}(X).
    \end{align*}
    The claim follows by combining the chain of inequalities.
    \end{proof}

\subsection{Proof of \Cref{proposition-filling-minimality-of-circles}}
    The proof is similar to that of \Cref{proposition-minimality-of-two-dimensional-planes}. Recall that $d_{\mathrm{int}}$ denotes the intrinsic length distance.

    \begin{proposition}\label{proposition-filling-area-of-a-circle}
        Let $Z$ be a complete metric space for which there exists an isometric embedding $( \mathbb{S}^1, d_{\mathrm{int}} ) \xhookrightarrow{} Z$. Then
        \begin{align*}
            \mathrm{FillArea}_{Z}( \mathbb{S}^1 ) \geq 2\pi.
        \end{align*}
    \end{proposition}
    
    \begin{proof}
        By \cite[Equation (10-1)]{lytchak-wenger-2018-intrinsic-structure-of-minimal-discs-in-metric-spaces}, the lowest value of $\mathrm{FillArea}_Z( \mathbb{S}^1 )$ is obtained when we isometrically embed $\mathbb{S}^1$ into $\ell^{\infty}$. (Notice that the same value is reached in any injective metric space containing $\mathbb{S}^1$). Next, it follows from \Cref{lemma-sobolev-to-lipschitz-filling-area-lipschitz-connectivity} that
        \begin{align*}
            \mathrm{FillArea}_{\ell^\infty}^{\mathrm{LIP}}( \mathbb{S}^1 ) 
            =
            \mathrm{FillArea}_{\ell^\infty}( \mathbb{S}^1 ).
        \end{align*}
        By \cite[Lemma 2.6]{lytchak-wenger-2016-regularity-of-harmonic-discs-in-spaces-with-quadratic-isoperimetric-inequalities}, the infimum $\mathrm{FillArea}_{\ell^\infty}^{\mathrm{LIP}}( \mathbb{S}^1 )$ is reached by considering the corresponding infimum over Lipschitz maps whose boundary trace is a constant-speed parametrization of $\mathbb{S}^1$. Combining this fact with \cite[Theorem 5.2]{Ivanov-lipschitz-metrics} and \cite[Theorem 2]{ivanov-systolic-ineq-proj} implies $\mathrm{FillArea}_{\ell^\infty}^{\mathrm{LIP}}( \mathbb{S}^1 ) \geq 2\pi$. The conclusion follows. 
    \end{proof}

    \begin{proof}[Proof of \Cref{proposition-filling-minimality-of-circles}]
    If $\mathcal{H}^{2}_{X}(X) = \infty$, then the inequality is trivial. Otherwise, we apply \Cref{proposition-uniformization-theorem} and consider $u \in \Lambda( \partial X, X )$ that is a uniform limit of homeomorphisms $\overline{\mathbb{D}} \to X$. Now, consider an isometric embedding $\iota \colon \mathbb{S}^1 \to \ell^{\infty}$, $x \mapsto ( d(x,x_j) )_{ j \in \N }$, where $(x_j)_j$ is a countable dense sequence in $\mathbb{S}^1$. By assumption, there exists a 1-Lipschitz homeomorphism $\phi \colon \partial X \to \mathbb{S}^1$. We find a $1$-Lipschitz extension $\Phi \colon X \to \ell^{\infty}$ of $\iota \circ \phi$, by applying McShane extensions componentwise. Then $\Phi \circ u \in \Lambda( \Gamma, \ell^{\infty} )$, where $\Gamma = \iota( \mathbb{S}^1 )$. Therefore,
    \begin{align*}
        \mathcal{H}^{2}_{X}(X) \geq \mathrm{Area}( \Phi \circ u ) \geq \mathrm{FillArea}_{ \ell^{\infty} }( \Gamma )
    \end{align*}
    by a similar argument as in the proof of \Cref{proposition-minimality-of-two-dimensional-planes}. The proof is complete after applying \Cref{proposition-filling-area-of-a-circle}.
    \end{proof}

\section{Systolic inequality for the torus}\label{section-systolic-inequality-torus}
      The goal of this section is to prove \Cref{theorem-main-result}. The main strategy of the proof consists of reducing the problem to the case of a flat Finsler torus. More precisely, given a metric torus $(X,d)$, we define a norm $\norm{\cdot}_\textup{st}$ on $\R^2$ through a suitable averaging process. This norm, called the stable norm, induces a flat torus $(\R^2/\Z^2, \norm{\cdot}_\textup{st})$ such that 
    \begin{equation}\label{eq-intro-systoly-equal}
        \sys(X,d) = \sys(\R^2/\Z^2, \norm{\cdot}_\textup{st}) 
    \end{equation}
    and
    \begin{equation}\label{eq-intro-area-ineq}
        \cH^2_d(X) \geq \textup{Area}_{\mathrm{BH}}(\R^2/\Z^2, \norm{\cdot}_\textup{st}).
    \end{equation}    
    Here, \eqref{eq-intro-systoly-equal} is quite direct from the construction. However, \eqref{eq-intro-area-ineq} is more involved and relies on minimality of two-dimensional planes, \Cref{proposition-minimality-of-two-dimensional-planes}. In the metric setting, the stable norm was introduced in \cite{Burago-peridoic-metric-soviet,burago-periodic-metric} and plays a critical role in the study of Finsler tori; see \cite{Burago-Ivanov-asymptotic-volume-finsler-tori, burago-ivanoc-kleiner,Balacheff-Gil-Moreno-de-Mora-finsler-tori}. We note that the stable norm also appears in the study of general Riemannian manifolds \cite{federer-stable-norm, gromov-book}.

\subsection{The stable norm}\label{sec: stable norm}

Here, we introduce the stable norm and some of its properties. We omit the proofs of statements that can already be found in \cite[Chapter 8.5]{Burago-Burago-Ivanov-2001-a-course-in-metric-geometry}.

Let $(X,d_X)$ be a length space that is homeomorphic to a two-dimensional torus. Then $X$ admits a universal cover $\pi \colon \mathbb{R}^2 \to X$ whose group of deck transformations is $\Z^2$ acting by translations. The distance $d_X$ on $X$ admits a lift
    \begin{equation}\label{eq:def-universal-cover-dist-torus}
        d( x, y ) = \inf\{ \ell_X( \pi \circ \gamma ) \colon \gamma \colon [0,1] \to \mathbb{R}^2 \text{ joins $x$ to $y$} \},
    \end{equation}
    where $\ell_{X}( \pi \circ \gamma )$ is the length of the curve $\pi \circ \gamma \colon [0,1] \to X$ with respect to the distance $d_X$. We emphasize that the natural action of $\mathbb{Z}^2$ on $\R^2$ through translations leaves $d$ invariant, which is to say that $d$ is \emph{$\mathbb{Z}^2$--periodic}. It follows that $(X,d_X)$ is isometric to $\R^2/\Z^2$ equipped with the quotient distance coming from $d$. Thus, from now on, we make no distinction between $(X,d_X)$ and $(\R^2/\Z^2,d)$.

    Fix a basepoint $x_0 \in \R^2$. For $z \in \Z^2$, the \textit{stable norm} of $z$ is defined as 
    \begin{equation}\label{eq-def-stable-norm}
        \norm{z}_\textup{st} = \lim_{k\to \infty} \frac{d(x_0,x_0+kz)}{k}.
    \end{equation}
    The existence of the limit follows directly from Fekete's subadditive lemma and the triangle inequality. The domain of $\norm{\cdot}_\textup{st}$ can be uniquely extended to $\Q^2$ by a positively homogeneous extension, and then to $\R^2$ by continuity.

    The stable norm is independent of the choice of basepoint, as an immediate consequence of the well--known bounded distance theorem:
    \begin{theorem}\label{theorem-bounded-distance-theorem}{\normalfont (Bounded distance theorem)}
        There exists $C>0$, depending only on $X$, such that for all $x \in ( \R^2, d )$ and every $z \in \Z^2$,
        $$\norm{z}_{\textup{st}} \leq d(x,x+z) \leq \norm{z}_\textup{st} + C.$$
    \end{theorem}
    The result can be proven using the following important lemma.
    
    \begin{lemma}\label{lemma-homotopy-class-minimizer}
        Let $X$ be a length space homeomorphic to a closed, oriented surface. Let $\gamma \colon \mathbb{S}^1 \to X$ be a length minimizer in its free homotopy class. Then $\gamma^{k} \colon \mathbb{S}^1 \to X$, where $z \mapsto \gamma( z^{k} )$, is a length minimizer in its free homotopy class. 
    \end{lemma}

    The lemma implies that, for every $z\in \Z^2$, its stable norm $\norm{z}_\textup{st}$ is equal to the length of the shortest curve in $X$ in the free homotopy class induced by $z$. In particular, 
    \begin{equation}\label{eq:characterization-stable-norm}
        \norm{z}_\textup{st} = \min_{x\in [0,1)^2} d(x,x+z)
    \end{equation}
    for every $z \in \Z^2$. This allows us to prove that the systole is stable when passing from the original distance to the stable norm.
    
    \begin{proposition}\label{proposition-limiting-systole}
        Let $d$ be a $\Z^2$--periodic metric on $\R^2$ and $\| \cdot \|_{\textup{st}}$ its stable norm. Then the systoles of $( \mathbb{R}^2/\mathbb{Z}^2, d )$ and $( \mathbb{R}^2/ \mathbb{Z}^2, \| \cdot \|_{\textup{st}} )$ are equal.
    \end{proposition}
    
    \begin{proof}
        Let $x \in \R^2$. Then
        $$\min_{z \in \mathbb{Z}^2\setminus\{0\}} d(x,x+z)$$
        is equal to the length of the smallest non-contractible closed curve in $( \mathbb{R}^2/\mathbb{Z}^2, d )$ starting and ending at the orbit of $x$. Recall that
        $$\sys (\mathbb{R}^2/\mathbb{Z}^2,d) = \min\{ \ell_d(\gamma) \colon \gamma \textup{ non-contractible closed curve in }\R^2/\Z^2\}.$$
        Therefore, by \eqref{eq:characterization-stable-norm} and the periodicity of the metric $d$, 
        \begin{align*}
            \sys (\mathbb{R}^2/\mathbb{Z}^2,d)  
            &= \min_{x \in [0,1)^2} \min_{z \in \Z^2\setminus\{0\}} d(x,x+z)
            \\
            &= \min_{ z \in \Z^2 \setminus \{0\} } \min_{ x \in [0,1)^2 } d(x,x+z)
            \\
            &= \min_{ z \in \Z^2 \setminus \{0\} } \norm{z}_{\st}.
        \end{align*}
        Clearly, the last term is equal to the systole on $(\R^2/\Z^2,\norm{\cdot}_\st).$
    \end{proof}

\subsection{Stability of the Hausdorff area}
    The goal of this subsection is to prove the following proposition. 
    
    \begin{proposition}\label{proposition-Hausdorff-area-drop}
        Let $d$ be a $\Z^2$-periodic length metric on $\R^2$, and $\norm{\cdot}_\textup{st}$ its stable norm. Then
        \begin{gather}\label{eq:h2-drops-in-prop}
            \cH^2_d( \R^2/\Z^2 ) \geq \cH^2_{\norm{\cdot}_\textup{st}}( \R^2/\Z^2 ).
        \end{gather}
    \end{proposition}

    \medskip

    For the convenience of the reader, we begin with a brief sketch of the proof, which follows the core strategy from \cite[Section 4.5]{Burago-Ivanov-asymptotic-volume-finsler-tori}.
    
    For $n\geq 1$, we consider $S_n\coloneqq [0,n)^2$. The idea is to use the minimality of two-dimensional planes, \Cref{proposition-minimality-of-two-dimensional-planes}, to compare $\cH^2_d(S_n)$ with $\cH_{\norm{\cdot}_\st}^2(S_n)$ and then use periodicity to obtain \eqref{eq:h2-drops-in-prop}. To apply \Cref{proposition-minimality-of-two-dimensional-planes}, we construct a linear isometric embedding $\Psi$ from $(\R^2,\norm{\cdot}_\textup{st})$ and a 1-Lipschitz map $\Phi$ from $(\R^2,d)$, both into the same Banach space $\mathbb{V}$, such that $\Phi(z) = \Psi(z)$ for every $z \in \Z^2$. It follows from the periodicity of $d$ that the distance between $\Psi(S_n)$ and $\Phi(S_n)$ is uniformly bounded. 
    We approximate each $S_n$ by a Jordan domain $X_n$ with rectifiable boundary and such that the area of $X_n$ and $S_n$ are comparable with an error term growing linearly in $n$. Next, we connect the boundaries $\Psi(\partial S_n)$ and $\Phi(\partial X_n)$ with a collar $C_n$. A key point in the construction of $X_n$ is that the area of $C_n$ increases linearly in $n$. This leads to a sequence of metric surfaces $X_n'$, each qualifying as a filling of $\Psi(\partial S_n)$ in the sense of \Cref{proposition-minimality-of-two-dimensional-planes}. Therefore,
    $$ \cH_d^2(X_n)+ \cH^2(C_n) = \cH^2(X_n') \geq \cH_{\norm{\cdot}_\st}^2(S_n).$$
    and hence,
    $$\limsup_{n \to \infty} \frac{\cH_d^2(S_n)}{n^2} = \limsup_{n \to \infty} \frac{\cH^2(X'_n)}{n^2} \geq  \limsup_{n \to \infty} \frac{\cH_{\norm{\cdot}_\st}^2(S_n)}{n^2}.$$
    Finally, \eqref{eq:h2-drops-in-prop} follows from the periodicity of $d$.
    \medskip

    The remainder of the section is devoted to the proof of \Cref{proposition-Hausdorff-area-drop}. We begin with the construction of the embedding and the $1$-Lipschitz map close to it.

    \begin{lemma}\label{lemma-embedding-via-calibrating-functions}
        There exists a Banach space $\mathbb{V}$, a linear isometric embedding 
        \linebreak $\Psi \colon (\R^2, \norm{\cdot}_\st) \to (\mathbb V,\norm{\cdot})$ and a $1$-Lipschitz map $\Phi \colon (\R^2,d) \to (\mathbb{V},\norm{\cdot})$ satisfying $\Phi(z) = \Psi(z)$ for every $z \in \Z^2$.
    \end{lemma}
    
    \begin{proof}
         Let $S$ be the unit sphere of the dual of $( \mathbb{R}^2, \| \cdot \|_{\textup{st}} )$. Below, $( \ell^{\infty}(S), \norm{\cdot} )$ denotes the space of bounded functions $S \to \R$ equipped with the supremum norm. As is clear, the map $\Psi \colon ( \R^2, \| \cdot \|_{\textup{st}} ) \to (\ell^{\infty}( S ),\norm{\cdot})$, where $x \mapsto ( h(x) )_{ h \in S }$, is a linear isometric embedding.
         Furthermore, it follows from \cite[Definition 4.1 and Proposition 4.3]{Burago-Ivanov-asymptotic-volume-finsler-tori} that for every $h \in S$, there exists a $1$-Lipschitz function $f_h\colon (\R^2,d) \to \R$ satisfying $f_h(z) = h(z)$ for every $z \in \Z^2$. Notice that \cite[Proposition 4.3]{Burago-Ivanov-asymptotic-volume-finsler-tori} is formulated for normed spaces. However, the same proof applies for $(\R^2,d)$. Finally, we define the $1$-Lipschitz map $\Phi \colon ( \R^2, d ) \to (\ell^{\infty}(S),\norm{\cdot}),$ where $x \mapsto ( f_h(x) )_{ h \in S }$. Clearly, $\Phi(z) = \Psi(z)$ for every $z \in \Z^2$.
    \end{proof}
    
    Next, we construct the Jordan domains.
    
    \begin{lemma}\label{lemma-curve-close-to-grid}
        Let $d$ be a $\Z^2$--periodic length metric on $\R^2$, and $\norm{\cdot}_\textup{st}$ its stable norm. There exist $N,C,D>0$ such that, for all $n\geq N$, there exists an injective curve $\gamma_n\colon(\partial S_n,\norm{\cdot}_\textup{st})\to (\R^2,d)$ with the following properties:
        \begin{enumerate}
            \item $\gamma_n$ is piecewise 1-Lipschitz on the edges of $\partial S_n$;
            \item $d(\gamma_n(x),x)<D$ for every $x \in \partial S_n$;
            \item the closure $X_n$  of the bounded component of $\R^2\setminus \gamma_n$ satisfies
            $$\cH^2_d(S_n) + Cn \geq \cH_d^2(X_n).$$
        \end{enumerate}
    \end{lemma}

    \begin{proof}
        Let $n> 2N$ for some large $N$, yet to be fixed. Let $\alpha,\beta\colon S^1\to ( \R^2/\Z^2, d )$ be shortest curves in their free homotopy classes, whose deck transformations correspond to $(1,0),(0,1)\in \Z^2$, respectively; existence follows from the Arzelá--Ascoli theorem using constant-speed parameterizations. Let $\alpha_0\colon \R\to \R^2$ be a $1$-periodic lift of $\alpha$ such that $\alpha_0(0)\in S_1$ and consider a vertical translation $$\alpha_n(t)=\alpha_0(t)+(0,n).$$ The lifts $\beta_0,\beta_n$ are defined the same way, except that $$\beta_n(t)=\beta_0(t)+(n,0)$$ instead. By \Cref{lemma-homotopy-class-minimizer}, it follows that the curves are injective. When $N$ is large enough, every $n > N$ is such that $\alpha_0$ and $\alpha_n$ have disjoint images. Obviously, we may also require that $\beta_0$ and $\beta_n$ have disjoint images.
    
        We form $\gamma_n$ by considering suitable subcurves of $\alpha_0$, $\beta_0$, $\alpha_n$, and $\beta_n$. To this end, we let $I_0 = \{ x \in \R \colon \alpha_0(x) \in \im \beta_1 \}$ and $I_n = \{ x \in \R \colon \alpha_0(x) \in \im\beta_n \}$. Notice that $I_n = I_0 + n$, and let $a \coloneqq \sup I_0$ and $b_n \coloneqq \inf I_n$. There exists $N \in \N$ such that every $n > N$ satisfies $L_n \coloneqq b_n-a > 0$. Let $s = \beta_0^{-1}(\alpha_0(a))$ and $t=\beta_n^{-1}(\alpha_0(b_n))$. We emphasize that, by the $1$-periodicity of the curves, the parameters $a$, $L_0 = L_n-n$, $s$, and $t$ are constant with respect to $n$. 
    
        Define $\gamma_n\colon(\partial S_n,\norm{\cdot}_\textup{st})\to (\R^2,d)$ as follows:
        \begin{gather*}
            \gamma_n(x,y)=
            \begin{cases}
                \alpha_0(a+L_n\cdot x/n),& \quad \text{when $y=0$},\\
                \alpha_n(a+L_n\cdot x/n), &\quad \text{when $y=n$},\\
                \beta_0(s+y), &\quad \text{when $x=0$, and}\\
                \beta_n(t+y), &\quad \text{when $x=n$}.
            \end{cases}
        \end{gather*}
        It is straightforward to check that $\gamma_n$ is well-defined, continuous, and injective. In particular, $\gamma_n$ is a homeomorphism. Furthermore, by \Cref{lemma-homotopy-class-minimizer}, it holds that $\norm{(1,0)}_\textup{st}=\ell(\alpha)$ and $\norm{(0,1)}_\textup{st}=\ell(\beta)$. This implies that $\gamma_n$ is 1-Lipschitz on each edge of $\partial S_n$ and hence, (1) follows.
        
        We claim that there exists a constant $D > 0$ such that for any $z\in \partial S_n$, we have $d(\gamma_n(z),z)<D$ for every $n> N$. To this end, consider a point $(x,0)\in \partial S_n$. We compute
        \begin{align*}
            d(\gamma_n(x,0),(x,0))&\leq
            d(\alpha_0(a+L_n\cdot x/n),\alpha_0(x))+d(\alpha_0(x),(x,0))\\
            &\leq \norm{(a+L_n\cdot x/n-x,0)}_\textup{st}+d(\alpha_0(x),(x,0))\\
            &\leq\norm{(a,0)}_\textup{st}+n^{-1}|x|\cdot|L_n-n|\norm{(1,0)}_\textup{st}+d(\alpha_0(x),(x,0)),
        \end{align*}
        where all terms in the final expression have upper bounds independent of $x$ and $n$. A similar computation yields the same conclusion for points on other edges of the domain. Thus (2) holds.        
        For proving (3), we may assume that $N>0$ is large enough such that, for every $n>2N$,
        $d(\partial[-N,n+N)^2,\partial S_n)>D.$
        For $n > 2N$, let $X_n$ be the closure of the bounded component of $\R^2\setminus \gamma_n$. We have $X_n\subset [-N,n+N)^2$ and therefore,
        $$\cH_d^2(X_n)-\cH_d^2(S_n)\leq \cH_d^2([-N,n+N)^2)-\cH_d^2(S_n),$$
        which, by the $\Z^2$--periodicity of $d$, can be rewritten as
        $$\cH_d^2(X_n)-\cH_d^2(S_n)\leq ((n+2N)^2-n^2)\cH_d^2(S_1).$$
        This further simplifies to
        $$\cH_d^2(X_n)-\cH_d^2(S_n)\leq (4Nn+4N^2)\cH_d^2(S_1),$$
        which completes the proof.
    \end{proof}

    Finally, we construct the metric surfaces $X'_n$. From now on, we adopt the notation from \Cref{lemma-embedding-via-calibrating-functions,lemma-curve-close-to-grid}.

    \begin{lemma}\label{lemma:metric-surfaces-construction}
        There exists $C'>0$ such that for every $n \geq N$ there exists a metric disk $X_n'$ satisfying 
        $$\cH^2(X_n) +C'n \geq \cH^2_d(X'_n),$$
        and a $1$-Lipschitz map $u_n \colon X'_n \to \mathbb{V}$ for which $u_n|_{ \partial X'_n }$ is a homeomorphism onto $\Psi( \partial S_n)$.
    \end{lemma}

    \begin{proof}
        Fix $n \geq N$. Let $x \in \partial S_n$ and $z$ be a closest point of $x$ in $\partial S_n\cap \Z^2$. Then, using that $\Psi(z) = \Phi(z)$ for every $z \in \Z^2$, we obtain
        \begin{align*}
            \norm{\Phi(\gamma_n(x))-\Psi(x)}&\leq\norm{\Phi(\gamma_n(x))-\Phi(z)}+\norm{\Psi(z)-\Psi(x)}\\
            &\leq d(\gamma_n(x),z)+\norm{z-x}_\textup{st}\\
            &\leq d(\gamma_n(x),x)+d(x,z)+\norm{z-x}_\textup{st}
            \leq D'
        \end{align*}
        for a constant $D' > 0$ independent of $n \geq N$ and $x$.
        
        Let $C_n \coloneqq  \partial S_n\times[0, D']$ be the cylinder equipped with the distance $d_{C_n}$ defined by
        $$d_{C_n}((x,s),(y,t)) = d_\textup{int}(x,y) + |s-t|;$$
        here $d_\textup{int}$ denotes the intrinsic distance on $\partial S_n \subset (\R^2,\norm{\cdot}_\st)$. We glue $C_n$ to $\partial X_n$ by identifying $(x,0)$ with $\gamma_n(x)$ for $x \in \partial S_n$. We denote the resulting space by $X'_n$ and equip it with the quotient distance $d_{X_n'}$. We recall the key properties of $d_{X_n'}$. The gluing map $(x,0) \mapsto \gamma_n(x)$ is $1$-Lipschitz since $\partial S_n$, seen as a subset of $C_n$, has the intrinsic distance. This allows us to define $d_{X'_n}$ in three pieces: Firstly, 
        \begin{equation}\label{equation-isometric-part}
            d_{X'_n}(x,y) =d(x,y)
        \end{equation}
        for all $x, y \in \partial X_n$, and also for $x,y \in X_n$. Thus, $X_n\xhookrightarrow{} X'_n$ is an isometry. Moreover, if $a = (x,s), b = (y,t) \in C_n$, let
        \begin{align*}
            D( a, b )
            =
            &\min_{ z, z' \in \partial S_n}\{
            d _{C_n}(a, (z,0) ) 
            + 
            d( \gamma_n(z), \gamma_n(z'))
            +
            d_{C_n}((z',0),b)
            \}.
        \end{align*}
        Then 
        \begin{equation}\label{equation-1-lipschitz-part}
            d_{X'_n}( (x,s), (y,t ) ) = \min\big\{ d_{C_n}((x,s),(y,t)), D( (x,s), (y,t) ) \big\}.
        \end{equation}
        This implies that the inclusion $C_n \xhookrightarrow{} X'_n$ is $1$-Lipschitz. Finally, if $x \in X_n$ and $(y,t) \in C_n$, then
        \begin{equation}\label{equation-jump-over-the-seam-part}
            d_{X'_n}(x,(y,t))
            =
            \min_{ z \in \partial X_n}
            d( x, \gamma_n(z))
            +
            d_{C_n}( (z,0), (y,t) ).
        \end{equation}
        The triangle inequality of $d_{X'_{n}}$ holds because the gluing map is $1$-Lipschitz.

        Now, let $H_n \colon (C_n,d_{C_n}) \to (\mathbb{V},\norm{\cdot})$, where $(x,s) \mapsto (1-\frac{s}{D'})\Phi(\gamma_n(x)) + \frac{s}{D'} \Psi(x)$. Notice that $H_n$ is $1$-Lipschitz. Indeed, for $(x,s),(y,t) \in C_n$, we have
        \begin{align*}
            \| H_n(x,s) - H_n(y,t) \|
            &\leq
            \| H_n(x,s) - H_n(y,s) \| + \| H_n(y,s) - H_n(y,t) \|
            \\
            &\leq d_{\mathrm{int}}(x,y) + 
            |s-t|\frac{\norm{\Phi(\gamma_n(y))-\Psi(y)}}{D'} 
            \\
            &\leq d_{\mathrm{int}}(x,y) + |s-t|.
        \end{align*}
        In the last inequality, we used that $\norm{\Phi(\gamma_n(y))-\Psi(y)}\leq D'$ for all $y \in \partial S_n$. Next, we consider the map $u_n \colon X'_n \to \mathbb{V}$ defined by
        \begin{equation*}
            u_n(x)
            =
            \left\{
            \begin{split}
                &\Phi(x), \quad&&\text{if $x \in X_n$,}
                \\
                &H_n(y,s),
                \quad&&\text{if $x = (y,s) \in C_n$.}
            \end{split}
            \right.
        \end{equation*}
        Since the embedding of $X_n$ into $X_n'$ is an isometry, recalling \eqref{equation-isometric-part}, it is immediate that the restriction of $u_n$ to $X_n$ is $1$-Lipschitz. Furthermore, by \eqref{equation-1-lipschitz-part} and as $H_n$ is 1-Lipschitz on $(C_n,d_{C_n})$, we conclude that the restriction of $u_n$ to $C_n$ is $1$-Lipschitz. The $1$-Lipschitz bound for points $x \in X_n$ and $(y,t) \in C_n$ follows from \eqref{equation-jump-over-the-seam-part}. Lastly, since $u_n|_{ \partial X'_n } = \Psi|_{ \partial S_n } $, it holds that $u_n|_{ \partial X'_n }$ is a homeomorphism. 

        Finally, regarding bounds on the two-dimensional Hausdorff measure of $X'_n$, we note that
        \begin{align*}
            \mathcal{H}^{2}( X'_n ) = \mathcal{H}^{2}_{ d }( X_n ) + \mathcal{H}^{2}_{d_{C_n}}( C_n ).
        \end{align*}
        Moreover, by the coarea inequality applied to the coordinate projection $(x,s) \mapsto s$ and the $\Z^2$--periodicity of $d$, we have
        $$\mathcal{H}^{2}_{d_{C_n}}( C_n )
            \leq
            \frac{4}{\pi}
            D' (1+D') \mathcal{H}_{\norm{\cdot}_\st}^{1}( \partial S_n ) \leq\frac{4}{\pi}
            D' (1+D') n\cH_{\norm{\cdot}_\st}^1(\partial S_1)<\infty.$$
        This completes the proof.
    \end{proof}

    We conclude this section with the proof of the stability of the Hausdorff measure.
    
    \begin{proof}[Proof of \Cref{proposition-Hausdorff-area-drop}]
        For $n \geq N$, let $X_n'$ and $u_n \colon X_n'\to \mathbb{V}$ be as in \Cref{lemma:metric-surfaces-construction}. Recall that $\Psi \colon (\R^2,\norm{\cdot}_\st) \to (\mathbb{V},\norm{\cdot})$ is a linear isometric embedding. Each $u_n$ is $1$-Lipschitz and a homeomorphism onto $\Psi(\partial S_n)$ when restricted to the boundary of the metric disk $X_n'$. Thus, it follows from \Cref{proposition-minimality-of-two-dimensional-planes} that
        $$\cH^2(X_n')\geq \cH^2_{\norm{\cdot}_\textup{st}}(S_n)$$
        for every $n \geq N$. The measure estimates in \Cref{lemma-curve-close-to-grid,lemma:metric-surfaces-construction} imply that there exists $c>0$, independent of $n$, such that
        $$\cH^2_d(S_n) + nc \geq \cH_{\norm{\cdot}_\textup{st}}^2(S_n)$$
        for every $n \geq N$. By additivity of the measure, and the fact that $d$ is $\Z^2$-periodic, we have $\cH_d^2(S_n) = n^2 \cH_d^2(S_1)$ for every $n \geq N$, and similarly $\cH_{\norm{\cdot}_\textup{st}}^2(S_n) = n^2 \cH_{\norm{\cdot}_\textup{st}}^2(S_1)$. Therefore,
        \begin{gather*}
            \cH^2_d(S_1)+\frac{cn}{n^2}\geq \cH^2_{\norm{\cdot}_\textup{st}}(S_1)
        \end{gather*}
        for every $n \geq N$. We conclude that $\cH^2_d( \R^2/\Z^2 ) \geq \cH^2_{\norm{\cdot}_\textup{st}}( \R^2/\Z^2 )$ by passing to the limit $n \to \infty$. 
    \end{proof}

\subsection{Proof of \Cref{theorem-main-result}}\label{sec: rigidity}
    We prove \Cref{theorem-main-result} in this subsection. For this, we need the following well-known systolic inequality for flat Finsler tori.

    \begin{theorem}\label{theorem-systolic-inequality-flat-finsler-tori}
        Let $(\mathbb{T}, \norm{\cdot})$ be a $2$-dimensional torus with a flat Finsler metric. Then,
        \begin{equation}\label{eq:sytolic-ineq-flat-finsler}
            \textup{Area}_\mathrm{BH}(\mathbb{T}, \norm{\cdot}) \geq  \frac{\pi}{4} \sys(\mathbb{T}, \norm{\cdot})^2.
        \end{equation}
        Equality holds for the flat metric corresponding to the supremum norm.
    \end{theorem}

    The result is essentially a reformulation of Minkowski's first theorem; see e.g. \cite[Theorem 5]{Balacheff-Gil-Moreno-de-Mora-finsler-tori}.

    \medskip

    \textbf{Minkowski’s first theorem}, 1896. Let $K \subset \R^2$ be a symmetric convex body such that $\interior(K) \cap \Z^2 = \{0\}$. Then its Lebesgue measure satisfies
    $$|K|\leq 4.$$
    
    Convex bodies that satisfy the previous inequality with equality are known as extremal bodies and have been well studied. An extremal body $K$ is either a parallelogram or a hexagon such that the integer translates of $\frac{1}{2}K$ form a tiling of $\R^2$; see e.g. \cite[Page 82]{geometry-of-numbers}. It follows that a flat Finsler torus $(\mathbb{T}, \norm{\cdot})$ satisfies \eqref{eq:sytolic-ineq-flat-finsler} with equality if and only if the unit disk of $\norm{\cdot}$ is an extremal body.

    \begin{proof}[Proof of Theorem~\ref{theorem-systolic-inequality-flat-finsler-tori}]
        Since $(\R^2/\Z^2,\norm{\cdot})$ is flat, we have
        $$\sys(\R^2/\Z^2,\norm{\cdot}) = \min_{z \in \Z^2\setminus \{0\}} \norm{z}.$$
        The systolic inequality is invariant under rescaling. Thus, we may assume that 
        $$\sys(\R^2/\Z^2,\norm{\cdot}) = \min_{z \in \Z^2\setminus \{0\}} \norm{z} = 1.$$
        In other words, the interior of the unit disk $K \subset \R^2$ of $\norm{\cdot}$ intersects $\Z^2$ only at the origin. Minkowski’s first theorem immediately implies 
        $$\textup{Area}_\mathrm{BH}(\mathbb{T}, \norm{\cdot}) = \int_{[0,1]^2} \frac{\pi}{|K|} \; dx \geq \frac{\pi}{4} =  \frac{\pi}{4} \sys(\R^2/\Z^2,\norm{\cdot})^2$$
        and equality holds if and only if $|K| = 4$.
    \end{proof}

    We can now complete the proof of the optimal systolic inequality for metric tori.
    \begin{proof}[Proof of \Cref{theorem-main-result}]
        Let $(X,d)$ be a length space homeomorphic to a torus. We write $\norm{\cdot}_\st$ for the stable norm on $\R^2$ induced by $d$ as described in Section \ref{sec: stable norm}.
        By \Cref{proposition-limiting-systole,proposition-Hausdorff-area-drop}, we have
        $$\sys(X,d) = \sys(\R^2/\Z^2, \norm{\cdot}_\st)$$
        and
        $$\cH_d^2(X) \geq \cH^2_{\norm{\cdot}_\st}(\R^2/\Z^2).$$
        Therefore, the optimal systolic inequality for flat Finsler tori, Theorem \ref{theorem-systolic-inequality-flat-finsler-tori}, implies
        $$\textup{Area}(X,d) = \cH^2_d(X) \geq\cH^2_{\norm{\cdot}_\st}(\R^2/\Z^2) \geq \frac{\pi}{4} \sys(\R^2/\Z^2, \norm{\cdot}_\st)^2 = \frac{\pi}{4}\sys(X,d)^2.$$
        The proof is complete.
    \end{proof}

\section{Systolic inequality for the real projective plane}\label{section-systolic-inequality-real-projective-plane}
    We prove \Cref{theorem-main-result-real-projective-plane} in this section. We recall the setup. Let $(X,d_X)$ be a length space homeomorphic to the real projective plane. We consider a locally isometric covering map $\pi \colon ( \S^2, d ) \to X$ whose covering group is generated by the antipodal map $g \colon \S^2 \to \S^2$, $x \mapsto -x$. The distance $d$ on $\S^2$ is defined similarly as in \eqref{eq:def-universal-cover-dist-torus}.
    Below, we show that a non-contractible curve of systolic length in $X$ lifts to an isometric image of a circle in the universal cover of $X$. This circle has a length twice that of the systole of $X$ and bounds a metric disk whose area is equal to that of $X$. The filling minimality of metric disks, \Cref{proposition-filling-minimality-of-circles}, then yields \Cref{theorem-main-result-real-projective-plane}. We note that the relation between the systole on the projective plane and the filling minimality of hemispheres was already used to prove the systolic inequality for projective planes equipped with a Finsler distance; cf. \cite{ivanov-systolic-ineq-proj}. We begin with the following observation.

     \begin{lemma}\label{lemma-description-of-the-systole}
        The systole of $X$ is equal to $\min_{ x \in \S^2 } d( x, g(x) )$.
    \end{lemma}
    
    \begin{proof}
        Let $x \in \S^2$. If $\alpha\colon I \to \S^2$ is a curve between $x$ and $g(x)$, then, $\pi\circ \alpha$ is a non-contractible curve in $X$. Thus,
        $$\min_{ x \in \S^2 }d( x, g(x) ) \geq \sys(X,d_X).$$
        On the other hand, if $\beta \colon I \to X$ is a non-contractible curve in $X$ starting and ending at $\pi(x)$, then $\beta$ has a lift to $(\S^2,d)$ connecting $x$ to $g(x)$. Hence, 
        $$\min_{ x \in \S^2 }d( x, g(x) ) \leq \sys(X,d_X).$$
        This completes the proof.
    \end{proof}

    Now, let $\gamma \colon [0,\ell] \to X$ be a unit-speed curve realizing the systole of $X$. Then there are two lifts $\gamma_+, \gamma_{-} \colon [0,\ell] \to (\S^2, d)$ of $\gamma$ along $\pi$ such that $\gamma_+ = g \circ \gamma_{-}$. Thus, we may form a closed curve $\theta \colon [0,2\ell] \to (\S^2, d)$ by concatenating $\gamma_{+}$ and $\gamma_{-}$. Clearly,
    \begin{equation}\label{equation-systole-double-loop}
        \ell(\theta) = 2 \ell(\gamma) = 2\ell= 2 \sys(X,d_X).
    \end{equation}
    We need the following lemma; for the statement, $\mathbb{S}^{1}(r)$ denotes the circle of radius $r$ and $d_{\mathrm{int}}$ the intrinsic length metric.
    
    \begin{lemma}\label{lemma-image-is-an-embedded-circle}
        There exists an isometric embedding $( \mathbb{S}^{1}( \ell/\pi ), d_{\mathrm{int}} ) \to ( \S^2, d )$ that parametrizes the image of $\theta$.
    \end{lemma}
    \begin{proof}
        We extend $\theta$ periodically to $\bar{\theta} \colon \R \to ( \S^2, d )$. It follows from the construction of $\theta$ that $g(\bar \theta(s))= \bar\theta(s+\ell)$ for all $s\in \R$. By \Cref{lemma-description-of-the-systole} and \eqref{equation-systole-double-loop}, this implies that for all $s\in \R$, the subcurve $\bar \theta|_{[s,s+\ell]}$ is a length minimizing curve between $\bar\theta(s)$ and $g(\bar\theta(s))$. Therefore, for all $s\in \R$ and $t\in (s,s+\ell]$, we have $d( \bar{\theta}(s), \bar{\theta}(t) ) = t-s$. In particular, there exists an isometry $( \mathbb{S}^{1}( \ell/\pi ), d_{\mathrm{int}} ) \to ( \S^2, d )$ as claimed.
    \end{proof}
    
    Let $\S_{+}$ be one of the complementary components of the image of $\theta$. Since $\pi|_{ \S_+ } \colon \S_+ \to X \setminus \pi(\theta)$ is a bijective local isometry, we have that $\cH_d^2(\S_+) $ is equal to the Hausdorff $2$-measure of $X$. Moreover, after rescaling the metric on $X$ such that the systole of $X$ is equal to $\pi$, we may suppose that $\S_+$ is a metric disk with boundary isometric to $( \mathbb{S}^{1}, d_{\mathrm{int}} )$. Hence, \Cref{proposition-filling-minimality-of-circles} implies
    $$\textup{Area}(X,d_X) = \cH^2_d(\S_+) \geq 2\pi = \frac{2}{\pi}\sys(X,d_X)^2.$$
    This concludes the proof of \Cref{theorem-main-result-real-projective-plane}.

\end{document}